\documentclass[letterpaper,11pt,twoside,keywordsasfootnote,addressatend,noinfoline]{article}
\usepackage{fullpage}
\usepackage[english]{babel}
\usepackage{amssymb}
\usepackage{amsmath}
\usepackage{theorem}
\usepackage{epsfig}
\usepackage{enumitem}
\usepackage{mathrsfs}
\usepackage{imsart}
\usepackage{color}

\theorembodyfont{\slshape}
\newtheorem{theorem}{Theorem}

\newtheorem{proposition}[theorem]{Proposition}
\newtheorem{lemma}[theorem]{Lemma}
\newtheorem{corollary}[theorem]{Corollary}

\newenvironment{proof}{\noindent{\scshape Proof.}}{\hspace*{2mm} $\square$}
\newenvironment{demo}[1]{\noindent{\scshape{Proof of #1.}}}{\hspace*{2mm}~$\square$}

\newcommand{\Z}{\mathbb{Z}}

\newcommand{\ind}{\mathbf{1}}
\newcommand{\ep}{\epsilon}

\DeclareMathOperator{\card}{card \,}
\DeclareMathOperator{\poisson}{Poisson \,}

\DeclareMathOperator{\geometric}{Geometric \,}

\def\lbd{\lambda}
\def\Lbd{\Lambda}
\def\beq{\begin{equation}}
\def\eeq{\end{equation}}
\def\bay{\begin{array}}
\def\eay{\end{array}}
\begin{document}

\begin{frontmatter}

\title     {Generalized stacked contact process \\ with variable host fitness}
\runtitle  {Generalized stacked contact process}
\author    {Eric Foxall\thanks{Research supported in part by an NSERC PDF Award} and Nicolas Lanchier\thanks{Research supported in part by NSA Grant MPS-14-040958.}}
\runauthor {Eric Foxall and Nicolas Lanchier}
\address   {School of Mathematical and Statistical Sciences, \\ Arizona State University, \\ Tempe, AZ 85287, USA.}

\begin{abstract} \ \
 The stacked contact process is a three-state spin system that describes the co-evolution of a population of hosts together with their symbionts.
 In a nutshell, the hosts evolve according to a contact process while the symbionts evolve according to a contact process on the dynamic subset of the lattice
 occupied by the host population, indicating that the symbiont can only live within a host.
 This paper is concerned with a generalization of this system in which the symbionts may affect the fitness of the hosts by either decreasing (pathogen) or
 increasing (mutualist) their birth rate.
 Standard coupling arguments are first used to compare the process with other interacting particle systems and deduce the long-term behavior of the
 host-symbiont system in several parameter regions.
 The mean-field approximation of the process is also studied in detail and compared to the spatial model.
 Our main result focuses on the case where unassociated hosts have a supercritical birth rate whereas hosts associated to a pathogen have a subcritical birth rate.
 In this case, the mean-field model predicts coexistence of the hosts and their pathogens provided the infection rate is large enough.
 For the spatial model, however, only the hosts survive on the one-dimensional integer lattice.

\end{abstract}



\end{frontmatter}

%%%%%%%%%%%%%%%%%%%%%%%%%%%%%%%%%%%%%%%%%%%%%%%%%%%%%%%%%%%%%%%%%%%%%%%%%%%%%%%%%%%%%%%%%%%%%%%%%%%%%%%%%%%%%%%%%%%%%%%%%%%%%%%%%%%%%%%%%%%%%%%%%%%%%%%%%%%%%%%%%%%%%%%%%%%%%%%%%%%%%%%%%%%%

\section{Introduction}
\label{sec:intro}

\indent The stochastic model considered in this paper is a generalization of the stacked contact process introduced in~\cite{court_blythe_allen_2012} and studied
 analytically in~\cite{lanchier_zhang_2015}.
 The stacked contact process is a spatial stochastic process based on the framework of interacting particle systems that describes the co-evolution of a population of hosts together
 with their symbionts.
 Individuals are located on the~$d$-dimensional integer lattice and interact with their nearest neighbors.
 The model assumes that the symbionts can only live in association with their host (obligate relationship) and are transmitted both vertically from associated hosts to their offspring
 and horizontally from associated hosts to nearby unassociated hosts.
 The stacked contact process~\cite{court_blythe_allen_2012, lanchier_zhang_2015} also assumes that all the hosts give birth and die at the same rate regardless of whether they are
 associated with a symbiont or not, meaning that the symbionts have no effect on the fitness of their host. \\
\indent This paper considers the natural generalization of the stacked contact process in which associated and unassociated hosts have different birth rates:
 symbionts that increase the birth rate of their host, and therefore have a beneficial effect, are referred to as a mutualists, whereas symbionts that decrease the birth rate of their
 host, and therefore have a detrimental effect, are referred to as pathogens.
 Formally, the state of the system at time~$t$ is a spatial configuration
 $$ \xi_t : \Z^d \longrightarrow \{0, 1, 2 \} $$
 where state~0 means empty, state~1 means occupied by an unassociated host, and state~2 means occupied by a host associated with a symbiont. Letting
 $$ \begin{array}{l} f_i (x, \xi) = (1/2d) \,\card \{y \in \Z^d : \sum_{j = 1, 2, \ldots, d} |x_j - y_j| = 1 \ \hbox{and} \ \xi (y) = i \} \end{array} $$
 be the fraction of nearest neighbors of vertex~$x$ which are in state~$i$, hosts and symbionts co-evolve according to the spin system whose transition rates at vertex~$x$ are given by
\begin{equation}
\label{eq:transitions}
  \begin{array}{rclcrcl}
   0 \ \to \ 1 & \hbox{at rate} & \lambda_{10} \,f_1 (x, \xi) & \quad & 1 \ \to \ 0 & \hbox{at rate} & 1 \vspace*{2pt} \\
   0 \ \to \ 2 & \hbox{at rate} & \lambda_{20} \,f_2 (x, \xi) & \quad & 2 \ \to \ 0 & \hbox{at rate} & 1 \vspace*{2pt} \\
   1 \ \to \ 2 & \hbox{at rate} & \lambda_{21} \,f_2 (x, \xi) & \quad & 2 \ \to \ 1 & \hbox{at rate} & \delta. \end{array}
\end{equation}
 The first four transition rates indicate that unassociated hosts give birth at rate~$\lambda_{10}$, hosts associated with a symbiont give birth at rate~$\lambda_{20}$ and,
 regardless of whether they are associated or not, all the hosts die at the normalized rate one.
 An offspring produced at~$x$ is sent to a vertex chosen uniformly at random among the nearest neighbors but the birth is suppressed when the target site is already occupied, which models
 competition for space.
 The offspring is always of the same type as its parent, indicating that the symbiont is always transmitted vertically.
 The process described by these four transitions is the multitype contact process~\cite{neuhauser_1992}.
 The effect of the symbiont on the host is modeled by the choice of the two birth rates: the symbiont is
 $$ \begin{array}{rcl}
     \hbox{a pathogen} & \hbox{when} & \lambda_{20} < \lambda_{10}  \vspace*{2pt} \\
    \hbox{a mutualist} & \hbox{when} & \lambda_{20} > \lambda_{10}. \end{array} $$
 The last two transitions describe the symbiont dynamics within the host population.
 The symbiont spreads to adjacent unassociated hosts at rate~$\lambda_{21}$, which corresponds to a horizontal transmission of the symbiont.
 Finally, hosts associated with a symbiont become unassociated at rate~$\delta$, which we simply call the recovery rate even when the symbiont is a mutualist.

\indent The stacked contact process~\cite{court_blythe_allen_2012, lanchier_zhang_2015} is obtained by setting~$\lambda_{20} = \lambda_{10}$.
 This corresponds to the neutral case in which the symbionts have no effect on the fertility of their hosts, i.e., all the hosts have the same birth rate.
 The analysis of this special case in~\cite{lanchier_zhang_2015} is somewhat facilitated by the fact that the process is attractive and monotone with respect to its parameters.
 This is true in certain cases when~$\lambda_{10} \neq \lambda_{20}$, although not in the cases that we consider. \\

% % % % % % % % % % % % % % % % % % % % % % % % % % % % % % % % % % % % % % % % % % % % % % % % % % % % % % % % % % % % % % % % % % % % % % % % % % % % % % % % % % % % % % % % % % % % % %

\noindent {\bf Mean-field approximation.}
 Before studying the spatial stochastic process, we first look at its non-spatial deterministic counterpart called mean-field approximation, consisting of a pair of coupled ordinary differential equations. 
 To derive it, consider a set of $N$ sites each of which can be either empty, occupied by an unassociated host, or occupied by an associated host. We suppose that each unassociated host attempts to give birth to an unassociated host onto a site chosen uniformly at random at rate $\lbd_{10}$, being successful if that site is empty. Similarly, each associated host attempts to give birth to an associated host at rate $\lbd_{20}$. Each host dies at rate $1$, while each symbiont dies (i.e. each associated host becomes an unassociated host) at rate $\delta$. Each associated host attempts to transmit the symbiont to a randomly chosen site at rate $\lbd_{21}$, being successful if the recipient is an unassociated host. Letting $U(t) = (U_0(t),U_1(t),U_2(t))$ denote the number of empty sites, unassociated hosts and associated hosts, respectively and rescaling to $u = U/N$, we have the Markov chain with transitions
 $$\bay{rclrcl}
  u \to u + N^{-1}(-1,1,0) & \hbox{at rate} & N \lbd_{10} \,u_0u_1, \quad &
  u \to u + N^{-1}(1,-1,0) & \hbox{at rate} & N u_1, \vspace*{2pt} \\
  u \to u + N^{-1}(-1,0,1) & \hbox{at rate} & N \lbd_{20} \,u_0u_2, \quad &
  u \to u + N^{-1}(1,0,-1) & \hbox{at rate} & N u_2, \vspace*{2pt} \\
  u \to u + N^{-1}(0,-1,1) & \hbox{at rate} & N \lbd_{21} \,u_1u_2, \quad &
  u \to u + N^{-1}(0,1,-1) & \hbox{at rate} & N \delta \,u_2. \eay$$
 This shows that~$u$ is a density-dependent Markov chain in the sense of~\cite{kurtz-ddmc}.
 Since the three densities add up to one, instead let~$u = (u_1, u_2)$.
 Writing as~$u^N$ to emphasize the dependence on~$N$, if~$\lim_{N \to \infty} u^N(0) = u$
 and $\ep,T > 0$, it follows from Theorem 2.2 in~\cite{kurtz-ddmc} that
 $$ \lim_{N\to\infty} P \left(\sup_{t \le T}|u^N(t) - u(t)| > \ep \right) = 0, $$
 where~$u(t)$ is the solution to the initial value problem $u(0) = u$ and
\begin{equation}
\label{eqMF}
 \begin{array}{rcl}
      u_1' & = & \lbd_{10} \,u_0 u_1 - u_1 + \delta u_2 - \lbd_{21} \,u_1 u_2 \vspace*{2pt} \\
      u_2' & = & \lbd_{20} \,u_0 u_2 - u_2 - \delta u_2 + \lbd_{21} \,u_1 u_2. \end{array}
\end{equation}
 It turns out to be more productive to study the proportion of hosts:
 $x_1 = u_1 + u_2$ and the proportion of hosts that are associated: $x_2 = u_2/x_1$.
 Letting $\lbd_a = \lbd_{20} - \lbd_{10}$ and $\lbd_b = -\lbd_a + \lbd_{21}$, after a bit of algebra, we obtain the system
\begin{equation}\label{eq:chgMF}
\begin{array}{rcl}
  x_1' & = & G_1 (x_1,x_2) = x_1 \,((\lbd_{10} + \lbd_a x_2)(1-x_1) - 1) \vspace*{2pt} \\
  x_2' & = & G_2 (x_1,x_2) = x_2 \,((\lbd_a +  \lbd_b x_1)(1-x_2) - \delta).
\end{array}
\end{equation}
 Define the set of interest $\Lambda = [0, 1]^2$ and
 $$ \Lbd_+ = \begin{cases} (0, 1] \times (0, 1) = \{(x_1, x_2) \in \Lbd : x_1>0, \ 0 <x_2<1 \} & \hbox{if} \ \delta = 0, \vspace*{2pt} \\
                           (0, 1] \times (0, 1] = \{(x_1, x_2) \in \Lbd : x_1,x_2 > 0 \} & \hbox{if} \ \delta > 0,\end{cases} $$ 
 which is obtained by removing invariant lines on the boundary of $\Lbd$. 
 In addition, let
 $$\bay{ll}
    p_0=(0,0),                         & \quad p_1 = (a_1, 0) = (1 - 1/\lbd_{10}, 0), \vspace*{2pt} \\
    p_2 = (a_2,1) = (1-1/\lbd_{20},1), & \quad p_3 = (0, a_3) = (0,1-\delta/\lbd_a). \eay$$
 Except for some corner cases, these are the only possible equilibria on the boundary of $\Lbd$.
 We define also two conditions on parameters:
 $$\begin{array}{cl}
   \hbox{(AinvU)}: & \lbd_{20}(1-a_1) + \lbd_{21}a_1 > 1 + \delta \vspace*{2pt} \\
   \hbox{(UinvA)}: & \lbd_{10}(1-a_2) - \lbd_{21}a_2 > 1. \end{array} $$
 The meaning of these two conditions is as follows:
\begin{itemize}
 \item (AinvU) stands for ``associated invades unassociated'', and is relevant if and only if~$\lambda_{10} > 1$, in which case it corresponds
               to parameter values for which a small introduction of associated hosts in a stable population of unassociated hosts leads to an increase
               in the proportion of associated hosts.
               Equivalently, $G_2 (p_1 + \ep e_2) > 0$ for small~$\ep > 0$, where~$e_2 = (0, 1)$. \vspace*{4pt}
 \item (UinvA) stands similarly for ``unassociated invades associated'', and is relevant if and only if the condition~$\lbd_{20}>1$
               and~$\delta=0$ is satisfied.
\end{itemize}
For $x \in \Lbd$ let $t\mapsto \phi(t,x)$ denote the solution to \eqref{eq:chgMF} with initial value $x$. 
The following result is proved in the companion paper~\cite{scp-mf}. We omit some details in the bistability case, since it is not the focus here.
\begin{theorem}\label{th:MF}
 The following six cases include all parameter values. \\
 There are two special cases.
\begin{enumerate}
 \item \underline{Redundant symbiont}. (RS)
        Suppose~$\max (\lbd_{10}, \lbd_{20}) > 1$ and~$\delta = \lbd_a = \lbd_{21} = 0$.
        For all~$x \in \lbd_+$, $\lim_{t \to \infty} \phi_1 (t, x) = a_1$ and~$t\mapsto \phi_2 (t, x)$ is constant. \vspace*{4pt}
 \item \underline{Bistability}. (B)
        There may be up to two locally stable equilibria.
        This occurs for some but not all parameter values satisfying~$\min (\lbd_a, \lbd_b) > 0$ and either
\begin{enumerate}[noitemsep,label=\roman*)]
 \item $\lbd_{10}\le 1 < \lbd_{20} < 1+\delta$ or \vspace*{2pt}
 \item $\lbd_{10}>1$ and (AinvU) does not hold.
\end{enumerate}
\end{enumerate}
 Suppose~(RS) and~(B) do not hold.
 Then, there exists~$\bar x \in \Lbd$ such that~$\lim_{t \to \infty} \phi(t, x) =\bar x$ for all~$x \in \Lbd_+$. 
 Assuming~(RS) and~(B) do not hold, four cases are possible.
\begin{enumerate}
\item \underline{Extinction}. (E) $\bar x = (0,\max(0,a_3))$ if~$\lbd_{10}\le 1$ and $\lbd_{20} \le 1+\delta$. \vspace*{4pt}
\item \noindent\underline{Survival and coexistence of associated and unassociated host}. (C)
%   $$ \begin{array}{rcl}
%      \max(0,\min(a_1,a_2)) < & \bar x_1 & < \max(a_1,a_2) \quad \hbox{and} \quad \\
%      \max(0,a_3) < & \bar x_2 & < 1, \quad \hbox{if} \end{array}$$
  $$ \max(0,\min(a_1,a_2)) < \bar x_1 < \max(a_1,a_2) \quad \hbox{and} \quad \max(0,a_3) < \bar x_2 < 1 $$
 in the following cases:
\begin{enumerate}[noitemsep]
\item $\delta>0$ and either
\begin{enumerate}[noitemsep,label=\roman*)]
\item $\lbd_{10}\le 1$ and $\lbd_{20}>1+\delta$, or \vspace*{2pt}
\item $\lbd_{10}>1$ and (AinvU) holds, or
\end{enumerate}
\item $\delta=0$ and either
\begin{enumerate}[noitemsep,label=\roman*)]
\item $\lbd_{10}\le 1$, $\lbd_{20}>1$ and (UinvA) holds, \vspace*{2pt}
\item $\lbd_{10}>1$, $\lbd_{20}\le 1$ and (AinvU) holds, or \vspace*{2pt}
\item $\min(\lbd_{10},\lbd_{20})>1$, (AinvU) holds and (UinvA) holds.
\end{enumerate}
\end{enumerate}
\item \underline{Survival of unassociated host only.} (UH) $\bar x = p_1$ if $\lbd_{10}>1$ and (AinvU) does not hold. \vspace*{4pt}
\item \underline{Survival of associated host only.} (AH) $\bar x = p_2$ if $\delta=0$, $\lbd_{20}>1$ and (UinvA) does not hold.
\end{enumerate}
\end{theorem}
 Before continuing we make a couple of observations concerning this result.
 First of all, (RS) occurs only for a single choice of parameters.
 Moreover, (B) occurs only when we have~$\lbd_a>0$ and~$\lbd_b>0$, which corresponds to a mutualist whose rate of spread through the population exceeds
 the increase it provides to the host birth rate.
 Aside from (RS) and (B), four behaviours are possible: the host goes extinct (and thus also the symbiont) (E), the host survives but not the symbiont (UH),
 both host and symbiont survive with coexistence of associated and unassociated hosts (C), or the host survives and the symbiont spreads completely through
 the host population (AH).
 In each case, the conditions are straightforward: extinction occurs if the birth rate is too low, hosts survive without symbiont if the symbiont cannot
 invade the host in equilibrium, etc. \\

% % % % % % % % % % % % % % % % % % % % % % % % % % % % % % % % % % % % % % % % % % % % % % % % % % % % % % % % % % % % % % % % % % % % % % % % % % % % % % % % % % % % % % % % % % % % % %

\noindent {\bf Spatial stochastic process.}
 We can show that the spatial stochastic process exhibits the four main regimes identified above for the mean-field equations.
 Notice that another way to describe these four regimes is as follows: both unassociated hosts (type 1) and associated hosts (type 2) can either survive or go extinct.
 Since, for an interacting particle system, there is more than one notion of survival, we distinguish the two notions that we use.
\emph{Single-site survival} of type~$i$ means that, for some initial configuration~$\xi$ with a positive and finite number of type~$i$ individuals (or if~$i=1$ and
 $\delta>0$, at least one occupied site),
 $$ P (\forall t > 0, \exists x : \xi_t (x) = i \,| \,\xi_0 = \xi)>0. $$
 The other notion of survival of type $i$ is that starting from a translation-invariant distribution that almost surely has an infinite number of type $i$ individuals, for all $x \in \Z^d$,
 $$\liminf_{t\to\infty}P(\xi_t(x)=i)>0.$$
 For the basic contact process discussed below, these two notions are known to coincide, a fact that follows from the model's self-duality (see~\cite{liggett_1985} for details).
 Note that when the recovery rate~$\delta > 0$, survival of associated hosts implies coexistence of associated and unassociated hosts. \\
\indent As noted above, the stacked contact process obtained by setting~$\lambda_{10} = \lambda_{20}$ has several nice properties including attractiveness and monotonicity.
 When $\lbd_{10} \ne \lbd_{20}$ it is still possible to have these properties, but only in certain cases.
 In this article we are not focused on the parameter regimes where attractiveness and monotonicity are present, except in the simpler subcases (covered in this
 section) where an easy comparison to an already-studied process can be used.
 For the sake of the interested reader who wishes to make a further study of this process, we note, without proof, some other cases, not considered in detail in
 this article, for which we have some monotonicity. \\
\indent Recall that a process is \emph{attractive} with respect to a partial order on configurations if for any~$\xi \leq \xi'$ and two copies~$\xi_t,\xi_t'$ of the
 process with~$\xi_0 = \xi, \xi_0' = \xi'$ there is a coupling of the two processes with the property that~$\xi_t \le \xi_t'$ for all~$t>0$.
 A process is \emph{monotone increasing} with respect to a parameter~$\rho$ if the above property holds when~$\xi_t,\xi_t'$ have respective parameter
 values~$\rho \le \rho'$, and \emph{monotone decreasing} if~$\rho \ge \rho'$.
 We focus on partial orders induced by a sitewise order on types, that is, $\xi \le \xi'$ if and only if~$\xi(x) \le \xi'(x)$ for all~$x \in \Z^d$.
 \begin{enumerate}
 \item $\lbd_{10}>\lbd_{20}$.
 \begin{enumerate}
 \item $\lbd_{21}=0$. Attractive for the order $0,2<1$ and monotone increasing in~$\lbd_{10},\delta$.
 \item $\lbd_{21}>0$. Not attractive for any order with $0<1$.
 \end{enumerate}
 \item $\lbd_{20}>\lbd_{10}$.
 \begin{enumerate}
 \item $\lbd_{20}>\lbd_{21}$.
 \begin{enumerate}
 \item $\delta>0$. Not attractive for any order with~$0<2$.
 \item $\delta=0$. Attractive for the order $0,1<2$ and monotone increasing in~$\lbd_{20},\lbd_{21}$.
 \end{enumerate}
 \item $\lbd_{20}=\lbd_{21}$. Type~2 sites give the basic contact process (described below) with birth rate~$\lbd_{20}$ and death rate~$1+\delta$.
 \item $\lbd_{21}>\lbd_{20}$. Attractive for the order $0<1<2$, monotone increasing in~$\lbd_{10},\lbd_{20},\lbd_{21}$ and monotone decreasing in $\delta$.
 \end{enumerate}
 \end{enumerate}
 Before getting into the detailed analysis, we first use couplings to compare the process with other popular interacting particle systems and collect some basic results.
 We start by comparing the process with the basic and the multitype contact processes using simple coupling techniques.
 We also show that the process inherits some of the properties of the forest fire model though, because of the lack of monotonicity, this does not simply
 follow from a standard coupling argument.

% % % % % % % % % % % % % % % % % % % % % % % % % % % % % % % % % % % % % % % % % % % % % % % % % % % % % % % % % % % % % % % % % % % % % % % % % % % % % % % % % % % % % % % % % % % % % %

\indent In the limiting case when the recovery rate~$\delta = \infty$, all the symbionts die instantaneously so the host dynamics reduces to the basic contact process
 $$ \begin{array}{rclcrcl}
     0 \ \to \ 1 & \hbox{at rate} & \lambda_{10} \,f_1 (x, \xi) & \quad & 1 \ \to \ 0 & \hbox{at rate} & 1. \end{array} $$
 There is a critical value~$\lambda_c \in (0, \infty)$ such that above~$\lambda_c$ the host population survives whereas at and below~$\lambda_c$ the population
 goes extinct~\cite{bezuidenhout_grimmett_1990}.
 This is in qualitative agreement with the mean-field equations.
 Assume from now on that the recovery rate is finite and let
%  $$ \xi_t^1 := \ind \{x \in \Z^d : \xi_t (x) \neq 0 \} \quad \hbox{and} \quad \xi_t^2 := \ind \{x \in \Z^d : \xi_t (x) = 2 \} $$
 $$ \xi_t^1 (x) = \ind \{\xi_t (x) \neq 0 \} \quad \hbox{and} \quad \xi_t^2 (x) = \ind \{x \in \Z^d : \xi_t (x) = 2 \} \quad \hbox{for all} \quad x \in \Z^d $$
 be the process that keeps track of the hosts and the process that keeps track of the hosts associated to a symbiont, respectively.
 The transitions for the first process satisfy
 $$ \begin{array}{rcr}
     0 \ \to \ 1 & \hbox{at rate at least} & \min \,(\lambda_{10}, \lambda_{20}) \,f_1 (x, \xi^1) \vspace*{2pt} \\
     0 \ \to \ 1 & \hbox{at rate at most}  & \max \,(\lambda_{10}, \lambda_{20}) \,f_1 (x, \xi^1) \end{array} $$
 while~$1 \to 0$ at rate one.
 In particular, this process can be coupled with the basic contact process described above to deduce that, for all~$x \in \Z^d$
and starting from a translation-invariant distribution with infinitely many 1s and 2s,
% \begin{equation}
% \label{eq:cp-1}
  \begin{eqnarray}
  \liminf_{t \to \infty} \,P (\xi_t (x) \neq 0) > 0 & \hbox{when} & \min \,(\lambda_{10}, \lambda_{20}) > \lambda_c \label{eq:cp-1a} \\
     \lim_{t \to \infty} \,P (\xi_t (x) \neq 0) = 0 & \hbox{when} & \max \,(\lambda_{10}, \lambda_{20}) \leq \lambda_c. \label{eq:cp-1b}
  \end{eqnarray}
 This follows from Theorem~III.1.5 in~\cite{liggett_1985}, which applies to general two-state spin systems, together with obvious inequalities relating the transition rates of our
 process and their counterpart for the basic contact process.
 Similarly, the transitions for the second process satisfy
 $$ \begin{array}{rcr}
     0 \ \to \ 1 & \hbox{at rate at least} & \min \,(\lambda_{20}, \lambda_{21}) \,f_1 (x, \xi^2) \vspace*{2pt} \\
     0 \ \to \ 1 & \hbox{at rate at most}  & \max \,(\lambda_{20}, \lambda_{21}) \,f_1 (x, \xi^2) \end{array} $$
 while~$1 \to 0$ at rate~$1 + \delta$, from which it follows that, for all~$x \in \Z^d$
and starting from a translation-invariant distribution with infinitely many 1s and 2s,
% \begin{equation}
% \label{eq:cp-2}
  \begin{eqnarray}
  \liminf_{t \to \infty} \,P (\xi_t (x) = 2) > 0 & \hbox{when} & \min \,(\lambda_{20}, \lambda_{21}) > (1 + \delta) \,\lambda_c \label{eq:cp-2a} \\
     \lim_{t \to \infty} \,P (\xi_t (x) = 2) = 0 & \hbox{when} & \max \,(\lambda_{20}, \lambda_{21}) \leq (1 + \delta) \lambda_c. \label{eq:cp-2b}
\end{eqnarray}
 The four parameter regions in~\eqref{eq:cp-1a}--\eqref{eq:cp-2b} are illustrated in the diagrams of Figures~\ref{fig:3D-host} and~\ref{fig:3D-symb}.
 So far, the behavior of the stochastic process agrees with the behavior of the mean-field model described in Theorem~\ref{th:MF}, if we think of the mean-field model as having~$\lambda_c = 1$ --
 note that $\min \,(\lambda_{20}, \lambda_{21}) > 1 + \delta$ implies (AinvU) holds.

% % % % % % % % % % % % % % % % % % % % % % % % % % % % % % % % % % % % % % % % % % % % % % % % % % % % % % % % % % % % % % % % % % % % % % % % % % % % % % % % % % % % % % % % % % % % % %

\indent Setting~$\lambda_{21} = \delta = 0$, the process reduces to the multitype contact process completely analyzed when the death rates are equal in~\cite{neuhauser_1992}.
 The transition rates become
 $$ \begin{array}{rclcrcl}
     0 \ \to \ 1 & \hbox{at rate} & \lambda_{10} \,f_1 (x, \xi) & \quad & 1 \ \to \ 0 & \hbox{at rate} & 1 \vspace*{2pt} \\
     0 \ \to \ 2 & \hbox{at rate} & \lambda_{20} \,f_2 (x, \xi) & \quad & 2 \ \to \ 0 & \hbox{at rate} & 1. \end{array} $$
 In this case, the type with the larger birth rate outcompetes the other type provided its birth rate is also strictly larger than the critical value of the single-type contact process.
 This result has first been proved in~\cite{neuhauser_1992} using duality techniques and again in~\cite{durrett_neuhauser_1997} using also a block construction in two dimensions to prove that the long-term behavior of the process
 is not altered by small perturbations of the parameters.
%  Even though the proof in~\cite{durrett_neuhauser_1997} focuses on the two-dimensional case, it easily extends to other spatial dimensions.
 In particular, using a similar perturbation argument, it can be deduced from~\cite[Propositions~3.1--3.2]{durrett_neuhauser_1997} that, for all~$x \in \Z^2$ and
 regardless of the initial configuration,
\begin{equation}
\label{eq:mcp}
  \lim_{t \to \infty} \,P (\xi_t (x) = 2) = 0 \quad \hbox{when} \quad \lambda_{10} > \lambda_{20} \ \hbox{and} \ \lambda_{21}, \delta \ \hbox{are small}.
\end{equation}
 The parameter region in~\eqref{eq:mcp} is shown in the top diagram of Figure~\ref{fig:3D-symb}.
 Also, it agrees with the mean-field equations.
 To see this, notice that if~$\lbd_{10} > \lbd_{20}$ are fixed and~$\lbd_{21} = \delta = 0$ then~(AinvU) reduces to~$\lbd_{20}/\lbd_{10} > 1$ which does not hold,
 so~(AinvU) still does not hold if~$\delta,\lbd_{21}>0$ are small.
 Since~$\lbd_{10}>\lbd_{20}$, (B) does not hold and~(RS) does not hold, so we have either~(E) or~(UH), depending on the values of~$\lbd_{10}$.
 In either case, Theorem \ref{th:MF} gives~$x_2 \to 0$ so associated hosts do not persist.

% % % % % % % % % % % % % % % % % % % % % % % % % % % % % % % % % % % % % % % % % % % % % % % % % % % % % % % % % % % % % % % % % % % % % % % % % % % % % % % % % % % % % % % % % % % % % %

\indent The forest fire model, also referred to as epidemics with recovery, is the three-state spin system with a cyclic dynamics described by the following three transitions:
 $$ \begin{array}{rclcrcl}
     0 \ \to \ 2 & \hbox{at rate} & \alpha \,f_2 (x, \xi) & \quad & 2 \ \to \ 1 & \hbox{at rate} & 1 \vspace*{2pt} \\
     1 \ \to \ 0 & \hbox{at rate} & \beta. \end{array} $$
 The three states are interpreted as~0 = alive, 2 = on fire and 1 = burnt, but can also be thought respectively as healthy, infected and immune in the context of epidemics.
 This process has been studied in~\cite{durrett_neuhauser_1991}, but note that we have interchanged the roles of the two states~1 and~2 to facilitate the comparison with our model.
 The main result in~\cite{durrett_neuhauser_1991} shows the existence of a critical value~$\alpha_c \in (0, \infty)$ such that, regardless of the value of~$\beta > 0$ and starting from
 a translation-invariant distribution with infinitely many 1s and 2s in two dimensions,
 $$ \liminf_{t \to \infty} \,P (\xi_t (x) = 2) > 0 \quad \hbox{when} \quad \alpha > \alpha_c. $$
 Because the dynamics is cyclic, basic couplings between the forest fire model and our process do not lead to any useful stochastic ordering between the two systems.
 However, the proof in~\cite{durrett_neuhauser_1991} easily extends to our process in a certain parameter region.
 Indeed, in addition to general geometrical properties and percolation results which are not related to the specific dynamics of the forest fire model, the key estimates in~\cite{durrett_neuhauser_1991}
 rely on the following two ingredients:
\begin{itemize}
 \item[(a)] The set of burning trees dominates its counterpart in the process with no regrowth ($\beta = 0$) provided both processes start from the same configuration. \vspace*{4pt}
 \item[(b)] In regions that have not been on fire for at least~$S$ units of time, the set of trees which are alive dominates a product measure with density~$1 - e^{- \beta S}$.
\end{itemize}
 Now, fix~$\delta \geq 0$, let~$\beta = 1 / (\delta + 1)$ and consider the spin system on the two-dimensional integer lattice whose dynamics is described by the five transitions
 $$ \begin{array}{rclcrcl}
     0 \ \to \ 2 & \hbox{at rate} & \beta \lambda_{20} \,f_2 (x, \xi) & \quad & 1 \ \to \ 0 & \hbox{at rate} & \beta \vspace*{2pt} \\
     1 \ \to \ 2 & \hbox{at rate} & \beta \lambda_{21} \,f_2 (x, \xi) & \quad & 2 \ \to \ 0 & \hbox{at rate} & \beta \vspace*{2pt} \\
                                                                           &&&& 2 \ \to \ 1 & \hbox{at rate} & \beta \delta. \end{array} $$
 Note that this is the process~\eqref{eq:transitions} with~$\lambda_{10} = 0$ slowed down by the factor~$\beta$.
 Alternatively, one can see this process as the forest fire model modified so that burnt trees can catch fire ($1 \to 2$) and trees on fire can spontaneously change
 to living trees ($2 \to 0$), skipping the burnt phase.
 In this process, trees burn for an exponential amount of time with rate~$\beta + \beta \delta = 1$ as in the original forest fire model.
 It follows that the domination property~(a) remains true: the set of burning trees in this new process dominates its counterpart in the forest fire model with
 no regrowth and in which the fire spreads by contact at rate~$\alpha = \beta \lambda_{20}$.
 Since the transition~$1 \to 0$ again occurs spontaneously at rate~$\beta$, the domination of the product measure~(b) remains true as well.
 In particular,
starting from a translation-invariant distribution with infinitely many 1s and 2s,
 $$ \liminf_{t \to \infty} \,P (\xi_t (x) = 2) > 0 \quad \hbox{when} \quad \lambda_{20} > (\delta + 1) \,\alpha_c. $$
 This holds for all~$\lambda_{21} \geq 0$.
 Since the proof in~\cite{durrett_neuhauser_1991} is based on a block construction, which supports small perturbations of the system, we also obtain coexistence in the process~\eqref{eq:transitions}
 under the same assumptions and provided~$\lambda_{10}$ is sufficiently small.
 In conclusion, in~$d = 2$
and starting from a translation-invariant distribution with infinitely many 1s and 2s,
\begin{equation}
 \label{eq:ffm}
   \liminf_{t \to \infty} \,P (\xi_t (x) = 2) > 0 \quad \hbox{when} \quad \lambda_{20} > (\delta + 1) \,\alpha_c \ \hbox{and} \ \lambda_{10} \ \hbox{is small}.
\end{equation}
 The parameter region in~\eqref{eq:ffm} is shown in the two diagrams at the bottom of Figures~\ref{fig:3D-host} and~\ref{fig:3D-symb}.

% % % % % % % % % % % % % % % % % % % % % % % % % % % % % % % % % % % % % % % % % % % % % % % % % % % % % % % % % % % % % % % % % % % % % % % % % % % % % % % % % % % % % % % % % % % % % %

\begin{figure}[t!]
\centering
\scalebox{0.40}{\input{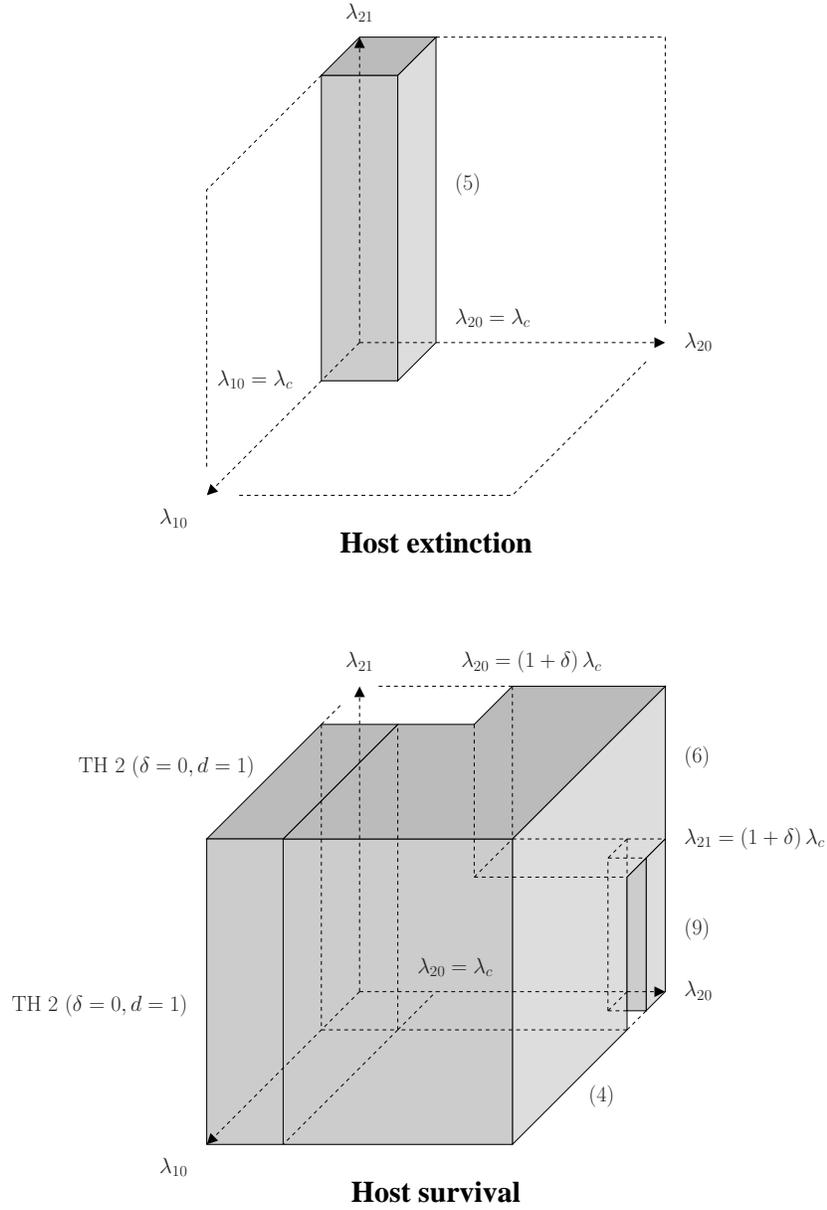}}
\caption{\upshape{Parameter regions in which the host dies out/survives.}}
\label{fig:3D-host}
\end{figure}

\begin{figure}[t!]
\centering
\scalebox{0.40}{\input{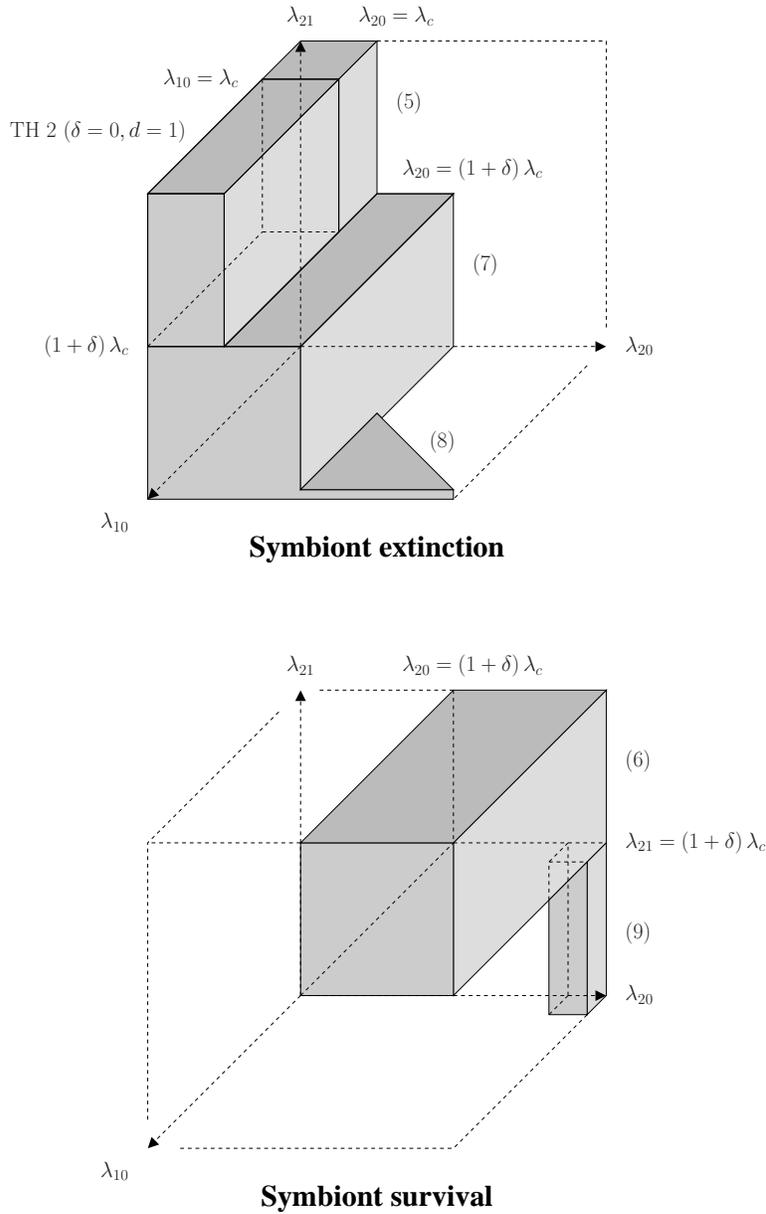}}
\caption{\upshape{Parameter regions in which the symbiont dies out/survives.}}
\label{fig:3D-symb}
\end{figure}

\indent We now focus on the parameter region where~$\lambda_{10} > \lambda_c > \lambda_{20}$ which is not covered by the previous comparison results.
 In this case, the symbiont is a pathogen.
 Standard coupling arguments to compare the host-pathogen system with the basic contact process imply that a population of healthy (unassociated) hosts survives whereas,
 if the recovery rate~$\delta = 0$, a population of infected (associated) hosts dies out.
 The long term behavior when starting with a mixture of healthy and infected hosts is not clear, and the main question is whether associated and
 unassociated hosts coexist.
 Theorem~\ref{th:MF} says that in the mean-field model they do coexist provided $\lambda_{21}$ is sufficiently large.
 For the spatial model, if the recovery rate is positive and the birth rate~$\lambda_{10}$ of healthy hosts is so large that the set of sites occupied by the hosts
 percolates, i.e., contains an infinite connected component, then it is expected that hosts and pathogens coexist provided the infection rate is reasonably large.
 But note that percolation of the set of hosts is only possible in~$d > 1$.
 In fact, we can prove that, in one dimension, even when~$\delta = 0$ and the infection rate and the birth rate of healthy hosts are very large, the pathogen is
 unable to survive.
\begin{theorem} --
\label{th:1D}
 Assume that~$\lambda_{10} > \lambda_c > \lambda_{20}$ and~$\delta = 0$ and~$d = 1$.
 Then, starting from any configuration with infinitely many vertices in state~1,
 $$ \liminf_{t \to \infty} \,P (\xi_t (x) = 1) > 0 \quad \hbox{and} \quad \lim_{t \to \infty} \,P (\xi_t (x) = 2) = 0 \quad \hbox{for all} \quad x \in \Z. $$
\end{theorem}
 The parameter region covered in Theorem~\ref{th:1D} is illustrated in the phase diagrams of Figures~\ref{fig:3D-host} and~\ref{fig:3D-symb}.
 %We point out that Theorem~\ref{th:1D} holds even when~$\lambda_{21} = \infty$, i.e., when the pathogen invades instantaneously nearby healthy hosts.
 Also, in addition to the statement of the theorem, our proof gives specific estimates on the rate of extinction of the pathogens and the rate of expansion of
 the healthy hosts.
 The first part of the proof shows that there exists a constant~$c > 0$ such that, uniformly in all initial configurations~$\xi_0$ with infinitely many~1s and
 for any site~$x$,
\begin{equation}
\label{eq:1D-part1}
  P (\sup \,\{t : \xi_t (y) = 2 \ \hbox{for some} \ y \ \hbox{such that} \ |y - x| \leq e^{ct} \} < \infty) = 1.
\end{equation}
 In other words, there exists a uniform (over all sites) exponentially growing (in time) neighborhood of any site which is eventually void of pathogens.
 To describe the long-term behavior of the healthy hosts, let~$\zeta_t$ denote the one-dimensional nearest-neighbor (supercritical) contact process with parameter~$\lambda_{10}$ starting from the all-one configuration.
 Also, let~$\alpha > 0$ denote the edge speed in this contact process as defined in~\cite{durrett_1980}; that is, starting from the initial
 configuration~$\xi_0 (x) = \ind \{x \le 0 \}$,
 $$\alpha = \lim_{t\to\infty}t^{-1}\sup\{x:\xi_t(x) \ne 0\}.$$
 Then, under the assumptions of the theorem, there exist
\begin{itemize}
 \item a random site~$X$ and an almost surely finite time~$T$ that depend on~$\xi$, \vspace*{2pt}
 \item site-valued processes~$\ell_t \leq r_t$ defined for $t \geq T$ and satisfying~$\ell_T = r_T = X$, \vspace*{2pt}
 \item a coupling of the processes~$\xi_t$ and~$\zeta_t$
\end{itemize}
 such that, $P$-almost surely,
\begin{equation}
\label{eq:1D-part2}
  \begin{array}{l}
  \lim_{t \to \infty} - \ell_t / t = \lim_{t \to \infty} r_t / t = \alpha \vspace*{4pt} \\ \hspace*{25pt}
  \hbox{and} \quad \xi_t (x) = \zeta_t (x) \ \hbox{for all} \ (x, t) \in [\ell_t, r_t] \times [T, \infty). \end{array}
\end{equation}
 In other words, as long as~$\xi$ has an infinite number of~1s then, $P$-almost surely, eventually there arises a stable population of~1s that behaves like the basic contact process on an interval that grows linearly in time.
 From~\eqref{eq:1D-part1}--\eqref{eq:1D-part2}, we also obtain a complete convergence theorem.
 Indeed, letting~$\nu$ denote the upper invariant measure of the contact process~$\zeta_t$ and~$\delta_0$ denote the measure that concentrates on the all-zero configuration, since the distribution of the contact process converges
 weakly to~$\nu$~\cite[Ch. VI]{liggett_1985}, we deduce the following for the distribution~$\mu_t$ of the process~$\xi_t$.
\begin{corollary} --
 Let~$\lambda_{10} > \lambda_c > \lambda_{20}$ and~$\delta = 0$ and~$d = 1$. Then, as~$t \to \infty$,
 $$ \mu_t \Rightarrow \rho \delta_0 + (1 - \rho) \nu \quad \hbox{where} \quad \rho = P_{\mu_0} (\{x : \xi_t (x) = 1 \}\neq \varnothing \ \hbox{for all} \ t > 0). $$
 In particular, all invariant measures are convex combinations of $\delta_0$ and $\nu$.
\end{corollary}

%%%%%%%%%%%%%%%%%%%%%%%%%%%%%%%%%%%%%%%%%%%%%%%%%%%%%%%%%%%%%%%%%%%%%%%%%%%%%%%%%%%%%%%%%%%%%%%%%%%%%%%%%%%%%%%%%%%%%%%%%%%%%%%%%%%%%%%%%%

\section{Graphical representation}
\label{sec:GR}

\indent Throughout the paper, we think of the process as being generated from a substructure, also called Harris' graphical representation~\cite{harris_1972}.
 The substructure consists of independent Poisson processes with appropriate rates attached to each vertex and oriented edge of the~$d$-dimensional integer lattice.
 The process is then constructed by assuming that, at the times of these Poisson processes, either a birth or an infection or a death or a recovery occurs whenever the configuration of the system at that time is compatible with the event.
 Table~\ref{tab:harris} shows how to construct the process using its substructure when~$\lambda_{10} > \lambda_{20}$, in which case the symbiont is a pathogen.

\begin{table}[t]
\begin{center}
\begin{tabular}{ccp{280pt}}
\hline \noalign{\vspace*{2pt}}
 rate                                  & symbol                             & effect on the process \\ \noalign{\vspace*{1pt}} \hline \noalign{\vspace*{6pt}}
 $ (\lambda_{10} - \lambda_{20}) / 2d$ & $x \overset{1}{\longrightarrow} y$ & birth at~$y$ when~$x$ is occupied by a healthy host and~$y$ is empty \\ \noalign{\vspace*{0pt}}
 $ \lambda_{20} / 2d$                  & $x \overset{2}{\longrightarrow} y$ & birth at~$y$ when~$x$ is occupied and~$y$ is empty \\ \noalign{\vspace*{0pt}}
 $ \lambda_{21} / 2d$                  & $x \overset{3}{\longrightarrow} y$ & infection at~$y$ when~$x$ is infected and~$y$ is occupied by a healthy host \\ \noalign{\vspace*{2pt}}
 1                                     & $\times$ at $x$                    & death at~$x$ when~$x$ is occupied \\ \noalign{\vspace*{2pt}}
 $\delta$                              & $\bullet$ at $x$                   & recovery at~$x$ when~$x$ is infected \\ \noalign{\vspace*{4pt}} \hline
\end{tabular}
\end{center}
\caption{\upshape{Graphical representation of the process when~$\lambda_{10} > \lambda_{20}$ (pathogen).
 The rates in the left column correspond to the different parameters of the independent Poisson processes, attached to either each oriented edge connected two neighbors (first three rows) or each vertex (last two rows).}}
\label{tab:harris}
\end{table}

\indent Note that the results in~\eqref{eq:cp-1a}--\eqref{eq:mcp} can be proved by coupling different processes using this graphical representation rather than~Theorem~III.1.5 in~\cite{liggett_1985}.
 For instance, the contact process~$\zeta_t^2$ with parameter~$\lambda_{20}$ can be constructed from the graphical representation in the table by assuming that births can only occur through type~2 arrows, while the contact process~$\zeta_t^1$
 with parameter~$\lambda_{10}$ can be constructed by assuming that births occur through both type~1 and type~2 arrows.
 Constructing our process and these two contact processes from this common graphical representation results in a coupling such that
 $$ \{x \in \Z^d : \zeta_t^2 (x) \neq 0 \} \subset \{x \in \Z^d : \xi_t (x) \neq 0 \} \subset \{x \in \Z^d : \zeta_t^1 (x) \neq 0 \} $$
 at all times~$t$ provided this holds at time zero.
 This shows~\eqref{eq:cp-1a}--\eqref{eq:cp-1b} when~$\lambda_{10} > \lambda_{20}$.
 This property when the inequality is reversed as well as~\eqref{eq:cp-2a}--\eqref{eq:cp-2b} and~\eqref{eq:mcp} are proved similarly by using other graphical representations which are designed based on the ordering of the parameters.

\section{Proof of Theorem~\ref{th:1D}}
\label{sec:1D}

\indent This section is devoted to the proof of~\eqref{eq:1D-part1}--\eqref{eq:1D-part2} which, together, imply Theorem~\ref{th:1D}. The first step is to obtain exponential bounds, in spacetime, on the the set of descendants (defined below in the natural way) of a type 2 individual, which is done in Proposition \ref{prop5}. To accomplish this we need two crucial observations described in a moment, together with an iterative or ``restart'' argument, and several estimates that build upon one another. We then show in Proposition \ref{prop7} that from any location, eventually the nearest type 2 individual will be at a distance which is exponentially far away as a function of time. This is then used to show that from an initial configuration with infinitely many type 1 sites, at least one of them will produce an set of type 1s growing linearly in time, none of which ever interact with a type 2, and completing the proof.\\

To obtain Proposition \ref{prop5}, the first crucial observation is the following asymmetry between sub- and super-critical contact processes. It is known (see for example \cite{griffeath_1981}) that for a (single-type) contact process on $\Z$ with half-line initial condition $\xi_0(x) = \ind(x \le 0)$ and defining the right edge $r_t:= \sup\{x:\xi_t(x) \ne 0\}$,
\begin{equation}
\begin{array}{rcll}
\label{eq:cont-asym}
\text{if} & \lambda>\lambda_c, & \text{eventually} & r_t \le Ct, \quad \text{while} \\
\text{if} & \lambda<\lambda_c, & \text{eventually} & r_t \le -e^{ct}, \end{array}
\end{equation}
where $c,C>0$ depend only on the value of $\lambda$ in each case. If $\lambda>\lambda_c$ we actually have $r_t/t \to \alpha(\lambda)$ but the above is the more pertinent fact here. In words, the invasion front of a supercritical contact process advances at most linearly, while the front for a subcritical contact process falls back exponentially fast. Thus if, in our process, we begin with $\xi_0(x) = \ind(x<0) + 2\ind(x>0)$, i.e., type 1 to the left of the origin and type 2 to the right, then if the right-hand boundary of type 1 and the left-hand boundary of type 2 do not meet within a short time, with high probability the two types will never interact, with the 2s vanishing rapidly while the 1s gradually advance. Naturally, this argument is also applicable if we start with a small patch of type 2s surrounded by type 1s.\\

The second observation is a comparison property that lets us reduce the study of the descendants of a type 2 site to the setting where there is a collection of 2s surrounded by 1s on either side. Namely, if in the initial configuration we replace \emph{all} 2s with 1s, then the resulting process has at least as many occupied sites as it did before. Given $\xi_t$ with $\xi_0=\xi$, if we define an auxiliary copy $\xi'$ on the same graphical representation, with initial configuration
$$\xi_0' = \ind \{\xi_0(y) \neq 0\},$$
then since $\lbd_{20} \le \lbd_{10}$ it follows that
 \beq\label{eq:comp-prop}
 \hbox{for all} \ t \ge 0, \ \{x:\xi_t(x)\ne 0\} \subseteq \{x:\xi_t'(x)=1\}.
 \eeq
Notice that the same is not true if we replace some but not all 2s with 1s, as can be seen by simple counterexamples. Next we define descendant and ancestor.
 Suppose~$\xi_s (x) = \xi_t (y) = i \ne 0$ for some~$x, y$ and~$s \leq t$.
 Then~$(y, t)$ is a \emph{descendant} of~$(x, s)$, and~$(x, s)$ is an \emph{ancestor} of~$(y,t)$ if either $(y,t)=(x,s)$ or if there are times and sites
 $$ s = t_0 < t_1 < \cdots < t_{k - 1} \leq t_k = t \quad \hbox{and} \quad x = x_1, x_2, \ldots, x_k = y $$
 such that the following two conditions hold:
\begin{itemize}
 \item For~$j = 1, 2, \ldots, k$, we have~$\xi_r (x_j) = i$ for all times~$r \in [t_{j - 1}, t_j]$. \vspace*{4pt}
 \item For~$j = 1, 2, \ldots, k - 1$, we have~$\xi_{t_j^-} (x_{j + 1}) \neq \xi_{t_j} (x_{j + 1}) = i$ as a result of a birth or infection event along the edge~$(x_j, x_{j + 1})$ at time~$t_j$.
\end{itemize}
For a set $S\subset \Z$ and $s\le t$, let
$$A(s,t;S) = \{y:(y,t) \ \hbox{is a descendant of} \ (x,s) \ \hbox{for some} \ x \in S\}.$$
Use the shorthand $A(s,t;x)$ for $A(s,t;\{x\})$ and $A_t(S)$ for $A(0,t;S)$, and for $i=1,2$ and a configuration $\xi$ let $S_i(\xi) =\{x:\xi(x)=i\}$. It follows from the definition of descendant that for $t\ge 0$ and $i=1,2$,
$$\{A(s,t;x) \colon x \in S_i(\xi_s)\} \quad \hbox{is a partition of} \quad S_i(\xi_t).$$
Thus to control $S_2(\xi_t)$, which is our goal, it's enough to get good bounds on $A_t(x)$ for $x \in S_2(\xi_0)$.
 Given $S$ disjoint from $S_1(\xi_0)$, let $A_t=A_t(S)$ and let~$\ell_t = \inf A_t$ and~$r_t = \sup A_t$ as well as
 $$ a_t = \sup \,\{x < \ell_t : \xi_t (x) \neq 0 \} \quad \hbox{and} \quad b_t = \inf\{x > r_t : \xi_t (x) \neq 0 \}. $$
 Also, let~$d_t = \ell_t - a_t$, ~$h_t = b_t - r_t$ and $m_t = \min(d_t,h_t)$.
 Figure~\ref{fig:descendant} gives an illustration of these quantities. 
 Notice that, since interactions are with nearest neighbours and $S\cap S_1(\xi_0)=\varnothing$ by assumption, it follows that $[\ell_t,r_t]\cap S_1(\xi_t) = \varnothing$ for $t\ge 0$, and a fortiori that
\beq\label{eq:semidesc}
A(s,t;[\ell_s,r_s]) = A_t \ \hbox{for} \ s\le t.
\eeq
Time intervals when $m_t=1$ we call \emph{invasion}, and when $m_t\ge 2$ we call \emph{struggle}. The basic recipe for controlling $A_t$ is to control the duration and extent of each invasion, and to show that each struggle results, with positive probability, in the rapid and total collapse of $A_t$.\\

We begin with struggle, which is the toughest to address -- in fact invasion will be surreptitiously taken care of in Proposition \ref{prop1}. In the next result we show that starting from a fixed initial configuration $\xi_0$, with positive probability, collapse of $A_t$ occurs before invasion, uniformly over finite intervals $S$ and $\xi_0$ such that $S_1(\xi_0)$ is disjoint from $S$. In addition, we show that if collapse occurs, then it is exponentially fast. This makes use of the comparison property \eqref{eq:comp-prop} as well as the asymmetry \eqref{eq:cont-asym}, and the fact that in the absence of invasion (that is, when $m_t>1$) the particles in $A_t$ do not interact with the particles outside $A_t$.
 Using these estimates and an iterative restarting argument, we can then prove Proposition~\ref{prop1}, which then easily leads to Propositions~\ref{prop3} and \ref{prop5}, at which point we have enough to tackle the proof of \eqref{eq:1D-part1} and \eqref{eq:1D-part2}.
 In the following proofs, $c$ and~$C$ are strictly positive constants such that the given statements hold for $c$ and all smaller values than $c$, or $C$ and all larger values than $C$, which will mean that $c, 1/C$ are allowed to decrease from step to step.
\begin{lemma} --
\label{lemma1}
There are~$p,c,C > 0$ so that, if $S\subset \Z$ is a finite interval, $S \cap S_1(\xi_0) = \varnothing$ and $m_0\ge 2$ then
$P(\tau = \infty) \ge p$ and for $t>0$,
$$P(r_s > r_0 - e^{cs}  \ \hbox{for some} \ s > t \ \hbox{and} \ \tau = \infty), \ P(t < \tau < \infty) \leq C \,e^{-ct}.$$
\end{lemma}
% \begin{figure}[t]
% \centering
% \scalebox{0.60}{\input{edges.pstex_t}}
% \caption{\upshape{Picture related to the proof of Lemma~\ref{lemma1}}}
% \label{fig:edges}
% \end{figure}
\begin{proof}
 %For an illustration of some of the random variables and processes introduced in the proof, we refer the reader to Figure~\ref{fig:edges}.
 Using a pair of independent substructures, define a pair of contact processes~$\zeta_t^1$ and~$\zeta_t^2$ with respective birth rates~$\lambda_{10}$ and~$\lambda_{20}$ and initial configurations 
 $$ \zeta_0^1 (x) = \ind(\xi_0(x) \ne 0, \ x \notin S) \quad \quad \hbox{and} \quad \quad  \zeta_0^2 (x) = \ind (\xi_0(x)=2, \ x \in S ).$$	
 Let
 $$\bay{rclrcl}
 \ell_t' &=& \inf\{y:\zeta_t^2(y) \ne 0\}, & r_t' &=& \sup\{y:\zeta_t^2(y) \ne 0\}\quad\hbox{and}\\
 a_t' &=& \sup\{y<\ell_t':\zeta_t^1(y) \ne 0\}, & b_t' &=& \inf\{y>r_t':\zeta_t^1(y) \ne 0\}.\eay$$
 Let $d_t' = \ell_t'-a_t'$, $h_t' = b_t'-r_t'$ and $m_t=\min(d_t',h_t')$ and let
 $$\tau' = \inf\{t:m_t'=1\}.$$
If a subset of $\Z$ does not contain, and is not adjacent to any occupied sites, that subset remains empty. Using this fact, there is a natural coupling of our process~$\xi_t$ with the two contact processes~$\zeta_t^1$ and~$\zeta_t^2$ obtained by defining~$\xi_t$ using both substructures up to time~$\tau$, then using only one substructure for all times~$t > \tau$.
Noting \eqref{eq:comp-prop}, this coupling has the property that for all~$t \le \tau$,
 $$A_t(\Z \setminus S) \subseteq \{x : \zeta_t^1 (x) = 1 \} \quad \hbox{and} \quad A_t(S) \subseteq \{x : \zeta_t^2 (x) = 1 \}.$$
It follows in particular that $a_t \le a_t'$, $\ell_t \ge \ell_t'$, $r_t \le r_t'$, $b_t \ge b_t'$ and thus $d_t \ge d_t',h_t \ge h_t'$ and $m_t \ge m_t'$ for $t\le \tau$. To simplify matters we note that $\tau = \tau_{\ell} \wedge \tau_r$, where
$$\tau_\ell = \inf\{t:d_t=1\} \quad \hbox{and} \quad \tau_r = \inf\{t:h_t=1\}.$$ 
\indent In the estimates that follow, $c,C$ and~$D$ are positive constants and $c, 1/C$ may get smaller from step to step.
 By monotonicity of the contact process, for~$i = 1,2$, the set of occupied sites~$\{x : \zeta_t^i (x) = 1 \}$ is dominated by the pure birth process in which particles do not die and give birth onto neighboring sites at rate~$\lambda_{i0}$, so the advance
 of type~$i$ into uncharted territory grows like at most~$\poisson (\lambda_{i0} \,t)$.
 In particular,
 $$ P (b_t' - r_t' < n) \leq P (d_0 - \poisson ((\lambda_{10} + \lambda_{20}) \,t) < n) \quad \hbox{for all} \quad n > 0, $$
 and applying a standard large deviations estimate, we get
\begin{equation}
\label{eq1}
  P (h_s' < 2 \ \hbox{for some} \ s \leq h_0 / (2 \,(\lambda_{10} + \lambda_{20}))) \leq C \,e^{-ch_0}.
\end{equation}
 Also, for each~$t > 0$,
\begin{equation}
\label{eq2}
  P (b_s' < b_0 - 2 \lambda_{10} \,t \ \hbox{for some} \ s \leq t) \leq e^{-ct}.
\end{equation}
 To control~$r_s'$, we use a known estimate at integer times, then a Poisson estimate at in-between times.
 From~\cite{griffeath_1981} and the assumption~$\lambda_{20} < \lambda_c$, for $t>0$,
\begin{equation}
\label{eq3}
  P (r_t' > r_0 - e^{ct}) \leq C \,e^{-ct}.
\end{equation}
 In addition, since the displacement in one unit of time is dominated by a Poisson random variable with parameter~$\lambda_{20}$, for any integer $k \geq 0$, we have
\begin{equation}
\label{eq4}
  P (r_s' - r_n' > k \ \hbox{for some} \ s \in [n, n + 1]) \leq C \,e^{-ck}.
\end{equation}
 Combining~\eqref{eq3} with $t=n$ and~\eqref{eq4}, we deduce that
\begin{equation}
\label{eq5}
  P (r_s' > r_0 - e^{cn} + n \ \hbox{for some} \ s \in [n, n + 1]) \leq C \,e^{-cn}
\end{equation} %\label{eq6a}
 Then, combining with~\eqref{eq2} evaluated at~$t = n + 1$, for each integer $n \geq 1$,
 \beq\label{eq:brbnd1}
  P (h_s' < h_0 + e^{cn} - (1 + 2 \,\lambda_{10})(n + 1) \ \hbox{for some} \ s \in [n, n + 1]) \leq C \,e^{-cn}.
  \eeq
 To deduce the first estimate, we distinguish two cases, where $D>0$ is a large enough constant. \vspace*{5pt} \\
\textbf{Case 1:~$m_0 > D$.} For one side of the argument, $h_0>D$ suffices -- an analogous argument applies to the other side assuming $d_0>D$. The bound on~$h_s'$ in \eqref{eq:brbnd1} is at least two for all~$n$.
Recalling that $h_t \ge h_t'$ for $t\le\tau$, then combining~\eqref{eq1} with~\eqref{eq:brbnd1} summed over~$n \geq \lfloor h_0 / (2 \,(\lambda_{10} + \lambda_{20})) \rfloor$,
\begin{equation}
\label{eq6}
  \begin{array}{rcl}
    P(\tau_r < \infty) &=& P(\inf_{t \ge 0}h_t \le 1) \le P (\inf_{t \ge 0}h_t' \le 1) \vspace*{4pt} \\
                     & \leq & Ce^{-ch_0} + \sum_{n \geq ch_0} \,C \,e^{-cn} \leq C \,e^{-ch_0} \quad \hbox{when} \quad h_0 > D, \end{array}
\end{equation}
where $c,C$ do not depend on $D$. By reflection invariance, the same holds for $\tau_\ell$. If $D$ is large enough that $Ce^{-cD}<1/2$ we find that $P(\tau = \infty) \ge 1-2Ce^{-cD} = \ep$ for some $\ep>0$.\\

\noindent\textbf{Case 2:~$m_0 \le D$.} Given $S$, let $E$ denote the event where, in one unit of time, there is a death at every site in $[\inf S-D,\inf S-1]\cup [\sup S+1,\sup S+D]$ and no birth onto any vertex in the same set. Since the birth rate onto any vertex is at most $2\lbd_{10}$ and the death rate at any site is $1$,
$$P(E) \ge (1-e^{-1})^{4D}(e^{-2\lbd_{10}})^{4D} = \delta>0,$$
with $\delta$ depending on $D$ but not on $S$. Note that on $E$, $\tau>1$ and $m_1>D$.
Using the Markov property and the previous result,
 $$P(\tau=\infty) \ge P(\tau>1 \ \hbox{and} \ m_1 > D)P(\tau=\infty \mid \tau>1 \ \hbox{and} \ m_1>D) \ge \delta \ep,$$
which gives the first statement with $p=\delta\ep>0$. \vspace*{5pt} \\
 Now, in~\eqref{eq5} above, for~$n \geq n_0$ for some~$n_0$, absorb~$n$ into~$-e^{cn}$ by decreasing~$c$, then increase~$C$ to account for~$n < n_0$.
 Then, for any~$t > 0$, summing~\eqref{eq5} over~$n \geq \lfloor t \rfloor$,
$$
  P (r_s' > r_0' - e^{-cs} \ \hbox{for some} \ s > t) \leq C \,e^{-ct}.
$$
 On the event~$\{\tau = \infty \}$, we have~$r_t \le r_t'$ for all~$t \geq 0$, and the second statement follows.
 Using the two bounds~\eqref{eq2} and~\eqref{eq3} above and noting $m_0\ge 2$,
 $$ P (h_t' < 2 + e^{ct} - 2 \,\lambda_{10} \,t) \leq C \,e^{-ct} $$
 and for~$t$ large enough, $2 + e^{ct} - 2 \, \lambda_{10} \, t \ge t$.
 Since $h_t \ge h_t'$ for $t\le \tau$, it follows that
 $$ P( h_t < t, \ t<\tau) \le C e^{-ct},$$
 and an analogous estimate applies to $d_t$. For~$t$ large enough, using the above and~\eqref{eq6},
 $$ \begin{array}{rcl}
      P (t < \tau < \infty) & = & P(m_t < t, \ t<\tau<\infty) + P(m_t \ge t, \ t<\tau<\infty) \\
      &\le& P(m_t < t, \ t<\tau) + P(t<\tau<\infty \mid m_t \ge t)\\
      &\le& P(h_t < t, \ t<\tau) + P(d_t<t, \ t<\tau) \\
      &+& P(t<\tau_r<\infty \mid m_t \ge t) + P(t<\tau_\ell<\infty \mid m_t \ge t)\\
	  &\le& 4C e^{-ct}. \end{array} $$
 This completes the proof.
\end{proof} \\

\begin{lemma}\label{lemma2}
Let $T=\inf\{t:m_t \ge 2\}$. There are positive constants~$c,c$ so that if $S\subset \Z$ is a finite interval and $S\cap S_1(\xi_0)=\varnothing$ then
$$P( T >t) \le Ce^{-ct}$$
\end{lemma}

\begin{proof}
Let $s_0=0$ and $s_1,s_2,\dots$ denote the times when either death occurs at $\ell_t-1$ or $r_t+1$, or infection occurs across either the edge $(\ell_t,\ell_t-1)$ or $(r_t,r_t+1)$. Then $T \le s_{2K}$ where
$$K = \inf\{k: \hbox{death occurs at} \quad (\ell_t-1,s_{2k-1}) \quad \hbox{and} \quad (r_t+1,s_{2k})\}.$$
Clearly, $K \preceq \geometric((1/2\lbd_{21})^2)$ and $\{s_{k+1}-s_k \colon k\ge 0\} \preceq \{\sigma_k \colon k\ge 0\}$, an i.i.d. sequence of $\exp(2)$ random variables. A routine estimate gives the result.
\end{proof}\\

Next we show the position of the rightmost pathogen in any $A_t$ goes to~$- \infty$ exponentially fast.
 This is the analog of the second estimate in~Lemma~\ref{lemma1} but dropping the condition $\tau=\infty$.
 This result is then used in the subsequent lemma to show that the probability that the rightmost pathogen moves~$n$ steps to the right of its initial position decays exponentially with~$n$.
 
\begin{proposition} --
\label{prop1}
 There are positive constants~$c,C$ so that if $S\subset \Z$ is a finite interval and $S\cap S_1(\xi_0)=\varnothing$ then
 $$ P(r_t > r_0 - e^{ct}) \leq C \,e^{-ct}. $$
\end{proposition}
\begin{proof}
Define recursively the two sequences of stopping times~$(\tau_i)_{i\ge 0}$ and~$(T_i)_{i\ge 0}$ by ~$\tau_0=0$, $T_0 = \inf\{t \ge 0:m_t=2\}$ and recursively for $i\ge 1$,
$$\bay{rcl}
\tau_i &=& \inf \,\{t > T_{i-1} : m_t = 1 \}, \\
T_i &=& \inf \,\{t > \tau_i : m_t = 2 \},
\eay$$
with the convention $\inf \varnothing = \infty$. Let $N = \inf\{i:\tau_i=\infty\}$. Recursively at each time $\tau_i$, applying the strong Markov property and noting \eqref{eq:semidesc}, then applying the first result of Lemma~\ref{lemma1} with $S = [\ell_{\tau_i},r_{\tau_i}]$ we find that
 $$ N \preceq - 1 + \geometric (p) $$
 where~$\preceq$ means stochastically smaller. Doing the same, but applying the third result of Lemma~\ref{lemma1},
 $$\{(\tau_i - T_{i-1}) \,\ind(N \geq i) \colon i>0\} \preceq \{\sigma_i \,\ind (N \geq i ) \colon i>0\}$$
 where~$\sigma_1,\sigma_2,\ldots$ is a sequence of independent, identically distributed random variables independent of the random variable~$N$ and such that
 $$ P (\sigma_i > t) = \min \,(C e^{-ct}, 1) \quad \hbox{for all} \quad t > 0. $$
Using Lemma \ref{lemma2}, the same holds for $\{(T_i-T_{i-1}) \ind (N \ge i) \colon i>0\}$. Letting~$n = \lceil t /4 E [\sigma_i]\rceil$ and applying a large deviations bound,

 \begin{equation}\label{equation-0}
\begin{array}{rcl}
      P (\tau_N > t) &\le& P((\tau_N-T_{N-1}) + \cdots + (\tau_1 - T_0) > t/2) \\
      &&+ P((T_{N-1}-\tau_{N-1}) + \cdots + (T_0 - \tau_0) > t/2)\\
      & \leq & 2P (\sigma_1 + \cdots + \sigma_N > t/2) \vspace*{4pt} \\
                       & \leq & 2P (N > n) + 2P (\sigma_1 + \cdots + \sigma_n > 2 n \,E [\sigma_i]) \leq 2C \,e^{-cn} \leq 2C \,e^{-ct}. \end{array}  
 \end{equation}

 Recall that~$\lambda_{20} \leq \lambda_{21}$ by assumption.
 Comparing the set of sites in state~2 to a pure birth process with no deaths and with birth to adjacent sites at rate~$\lambda_{21}$, a large deviations estimate gives~$c, C > 0$ so that
\begin{equation}
\label{equation-1}
  P (r_t > r_0 + 2 \,\lambda_{21} \,t) \leq C \,e^{-ct} \quad \hbox{for all} \quad t > 0.
\end{equation}
 For any~$c > 0$, there is~$t_0$ so that~$\lambda_{21} \,t < e^{ct} - e^{ct/2}$ for all~$t > t_0$, in which case
\begin{equation}
\label{equation-2}
  \begin{array}{rcl}
    P (r_t > r_0 - e^{ct/2}) & \leq & P (r_{t/2} > r_0 + \lambda_{21} \,t) \vspace*{4pt} \\
                               &      & \hspace*{10pt} + \ P (r_t > r_0 - e^{ct/2} \ \hbox{and} \ r_{t/2} \leq r_0 + \lambda_{21} \,t) \vspace*{4pt} \\
                               & \leq & P (r_{t/2} > r_0 + \lambda_{21} \,t) + P (r_t > r_{t/2} - \lambda_{21} \,t - e^{ct/2}) \vspace*{4pt} \\
                               & \leq & P (r_{t/2} > r_0 + \lambda_{21} \,t) + P (r_t > r_{t/2} - e^{ct}) \end{array}
\end{equation}
 for all~$t > t_0$. On the other hand,
\begin{equation}
\label{equation-3}
   P (r_t > r_{t/2} - e^{ct}) \le P (\tau_N > t/2) + P(r_t > r_{t/2}- e^{ct}, \ \tau_N \le t/2)
\end{equation}
Letting $\tau(t) = \inf\{s>t:m_s=1\}$, the event $\tau_N \le t/2$ is equivalent to $\tau(t/2)=\infty$. Applying the second result of Lemma \ref{lemma1} with $S=[\ell_{t/2},r_{t/2}]$ we find that
$$P(r_t > r_{t/2} - e^{ct} \ \hbox{and} \ \tau(t/2)=\infty) \le Ce^{-ct}.$$
 Combining with equations~\eqref{equation-0}--\eqref{equation-3} gives the desired estimate when~$t > t_0$.
 If~$t\leq t_0$ then, after increasing~$C$ if necessary, the estimate holds for all values of~$t$.
\end{proof}

\begin{proposition} --
\label{prop3}
 There is~$C > 0$ so that if $S\subset \Z$ is a finite interval and $S \cap S_1(\xi_0) = \varnothing$ then
 $$ P (r_t > r_0 + n \ \hbox{for some} \ t \geq 0) \leq C \,e^{-cn}.$$
\end{proposition}
\begin{proof}
 Comparing to a pure birth process as above,
 $$ P (r_s > r_0 + n \ \hbox{for some} \ s \leq m_0) \leq C \,e^{-cn} \quad \hbox{for} \quad m_0 = \lfloor n / 2 \lambda_{21} \rfloor. $$
If $n \ge n_0 = \sup_{m\ge 0}- e^{cm} + m$ then using Proposition~\ref{prop1} and large deviations for the Poisson distribution with parameter~$\lambda_{21}$,
 $$ \begin{array}{l}
      P (r_s > r_0 + n \ \hbox{for some} \ s \in [m, m + 1]) \vspace*{4pt} \\ \hspace*{20pt} \leq \
      P (r_m > r_0 - e^{-cm}) + P (r_s > r_m + m \ \hbox{for some} \ s \in [m, m + 1]) \leq C \,e^{-cm}. \end{array} $$
 Summing over~$m \geq m_0$ gives the desired estimate for $n\ge n_0$ -- for $n<n_0$ increase $C$ if necessary.
\end{proof}
 
% % % % % % % % % % % % % % % % % % % % % % % % % % % % % % % % % % % % % % % % % % % % % % % % % % % % % % % % % % % % % % % % % % % % % % % % % % % % % % % % % % % % % % % % % % % % % %

\begin{figure}[t]
\centering
\scalebox{0.60}{\input{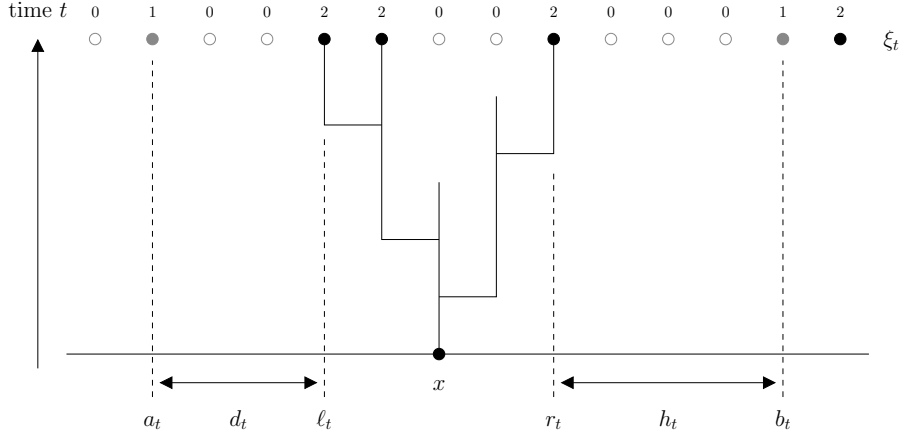}}
\caption{\upshape{Picture related to the proof of Proposition~\ref{prop5}}}
\label{fig:descendant}
\end{figure}

We can now show the set of descendants is exponentially bounded in both space and time.

\begin{proposition} --
\label{prop5}
There are~$c, C > 0$ so that if $x \in \Z$ and $\xi_0(x)=2$,
 $$ P(A_t(x) \neq \varnothing) \leq C \,e^{-ct} \quad \hbox{and} \quad P(A_t(x) \notin [x - n, x + n] \ \hbox{for some} \ t \geq 0) \leq C \,e^{-cn} $$
\end{proposition}

\begin{proof}
Defining $a_t,\ell_t$ etc. with $S=\{x\}$, $\ell_0=r_0=x$ and $A_t(x)\ne\varnothing$ is equivalent to $\ell_t \le r_t$. Using Proposition \ref{prop1}, reflection invariance and a union bound, 
$$\bay{rcl}
P(\ell_t > r_t) &\ge& P(\ell_t \ge \ell_0 + e^{ct} \ \hbox{and} \ r_t \le r_0-e^{ct}) \\
&\ge & 1 - (P(\ell_t< \ell_0+e^{ct}) + P(r_t > r_0-e^{ct}) \ge 1-2Ce^{-ct},\eay$$
and the first result follows by taking the complement. The second result follows in the same way, except using Proposition \ref{prop3}.
\end{proof} \\ \\

 Next, we use Proposition~\ref{prop5} to prove~\eqref{eq:1D-part1}, which says that in an exponentially growing neighborhood of any site, eventually there are
 no sites in state~2.
 We also record an exponential estimate.
Note the change in the definition of $\ell_t^x,r_t^x$.
\begin{proposition} --
\label{prop7}
 For a site~$x \in \Z$, let
 $$ \ell_t^x = \sup \,\{y \leq x : \xi_t(y) = 2 \} \quad \hbox{and} \quad r_t^x = \inf \,\{y \geq x : \xi_t (y) = 2 \}. $$
%  with the convention $\ell_t^x=-\infty, r_t^x = \infty$ if either set is empty.
 Then, there exist~$c, C > 0$ such that, for any~$\xi_0$ and~$t_0$,
 $$ P(\ell_t^x > x - e^{ct} \ \hbox{or} \ r_t^x < x + e^{ct} \ \hbox{for some} \ t > t_0) \leq C \,e^{-ct_0}. $$
 Also, there exists~$c > 0$ so that, for any~$\xi_0$ and any site~$x$,
 $$ P(\sup \,\{t : \xi_t (y) = 2 \ \hbox{for some} \ y \ \hbox{such that} \ |y - x| \leq e^{ct} \} < \infty) = 1. $$
\end{proposition}
\begin{proof}
 Throughout this proof, $y$ refers to a site which is initially in state~2.
 Let~$c, C$ be two constants as in Proposition~\ref{prop5}, so that
 $$ P(A_n(y) \neq \varnothing) \leq C \,e^{-cn} \quad \hbox{for all} \quad y, \xi \ \hbox{and} \ n. $$
 Using a union bound over~$y \in [x - e^{cn/2}, x + e^{cn/2}]$ and that~$A_t(y) = \varnothing$ is an absorbing property,
 $$ \begin{array}{l}
      P (A_t(y) \neq \varnothing \ \hbox{for some} \ y \ \hbox{such that} \ \xi_0 (y) = 2 \vspace*{4pt} \\ \hspace*{80pt}
                     \hbox{and} \ |y - x| \leq e^{cn/2} \ \hbox{and some} \ t \geq n) \leq C \,e^{-cn/2}. \end{array} $$
 Let~$n_0$ be such that~$n \geq n_0$ implies~$e^{cn/2} - e^{c (n + 1)/4} > n$.
 If~$|y - x| = \lceil e^{cn/2} \rceil + m$ for integer~$m \geq 0$ and all $n \geq n_0$, then we have
 $$ P (|\sup A_t(y)  - x| \leq e^{c (n + 1)/4} \ \hbox{for some} \ t \geq 0) \leq C \,e^{-c (n + m)} $$
 Taking~$C$ larger if necessary makes the previous inequality true also for all~$n < n_0$.
 Then, taking a union bound and increasing~$C$ at the last step gives
 $$ \begin{array}{l}
      P (|\sup_t |\sup A_t(y) - x| \leq e^{c (n + 1)/4} \ \hbox{for some} \ y \vspace*{4pt} \\ \hspace*{50pt}
                \hbox{such that} \ |y - x| > e^{cn/2} \ \hbox{and} \ \xi_0 (y) = 2) \leq C (1 - e^{-c}) \,e^{-cn} = C \,e^{-cn}. \end{array} $$
 Combining the estimates gives
 $$ P (\xi_t (y) = 2 \ \hbox{for some} \ y \in [x - e^{ct/4}, x + e^{ct/4}] \ \hbox{and} \ t \in [n, n + 1]) \leq C \,e^{-cn/2}. $$
 Given~$t_0$, summing over~$n \geq \lfloor t_0 \rfloor$ gives the first statement. 
 Summing over all~$n$ and using the Borel-Cantelli lemma finishes the proof.
\end{proof} \\ \\

 In what follows, a \emph{mark} refers to a Poisson point in the graphical representation.
 Marks are named by their effect on the target site, so for example, a~$0 \to 1$ mark is an edge mark at some time~$t$ along a directed edge~$(x, y)$ such that, if~$\xi_{t^-} (x) = 1$ and~$\xi_{t^-} (y) = 0$ then~$\xi_t (y) = 1$.
 Also, a death mark refers to a~$\star \to 0$ event, while a birth mark is a~$0 \to \star$ event, where~$\star \neq 0$.
 Note that since the graphical representation consists of at most a countably infinite number of Poisson point processes, with probability one, no two marks occur at the same time.\\
 
 Next we use Proposition \ref{prop7} to produce an interval that grows linearly in time and is devoid of pathogens.

\begin{lemma} --
\label{lemma7}
 For any~$\mu > 0$, there are~$c, C > 0$ so that
 $$ P (\xi_t (x) = 2 \ \hbox{for some} \ |x| \leq n / 2 + \mu t) \leq C \,e^{-cn} $$
 for all~$\xi_0$ such that~$\xi_0(x) \neq 2$ for all~$|x| \leq n$.
\end{lemma}
\begin{figure}[t]
\centering
\scalebox{0.60}{\input{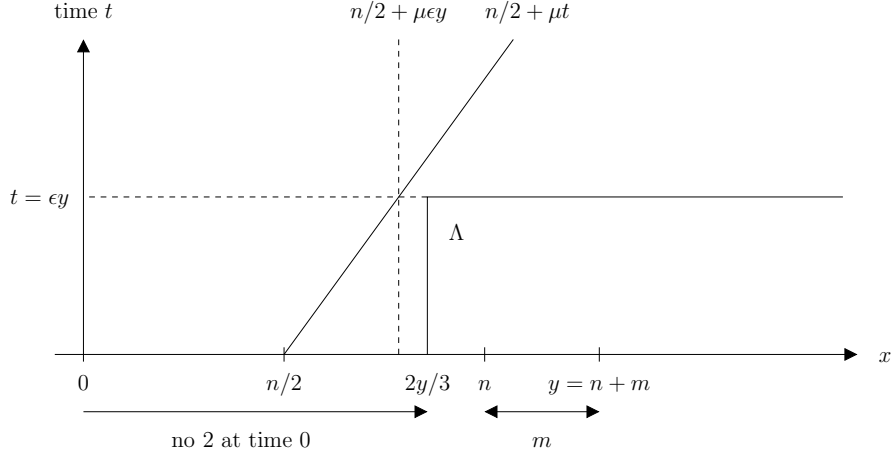}}
\caption{\upshape{Picture related to the proof of Lemma \ref{lemma7}}}
\label{fig:rectangle}
\end{figure}
\begin{proof}
 Let~$y = n + m$ with~$m > 0$; a similar argument applies to~$-y$.
 Using Proposition~\ref{prop5} gives the existence of constants~$c, C > 0$ so that, for all $\ep > 0$,
 $$ \begin{array}{l} P (A_{\ep y}^y \neq \varnothing  \ \hbox{or} \ \inf_t A_t^y < 2y/3) \leq C \,e^{-c \ep y}. \end{array} $$
 On the complement of the above event, the descendants of~$(y, 0)$ are contained in
 $$ \Lambda := \{(x, t) : x \geq 2y/3 \ \hbox{and} \ 0 \leq t \leq \ep y \}. $$
 A quick sketch (see Figure~\ref{fig:rectangle}) shows that the rectangle~$\Lambda$ is disjoint from the set
 $$ \{(x, t) : t \geq 0 \ \hbox{and} \ |x| \leq n / 2 + \mu t \} $$
 provided the top left corner of~$\Lambda$ lies to the right of the line~$x = n / 2 + \mu t$, which is the condition
 $$ n / 2 + \mu \ep y < 2y / 3. $$
 Since~$n \leq y$, this condition is satisfied if~$\ep < 1/ 6 \mu$.
 Summing over~$m > 0$ for both~$y = n + m$ and~$y = - n - m$ then gives the desired result.
\end{proof} \\ \\
 A related notion to the descendants is the \emph{cluster}, that we need only define for type~1, as follows.
 Suppose~$\xi_s (x) = \xi_t (y) = 1$ for some~$x, y$ and~$s \leq t$.
 Then, we say that~$(y, t)$ belongs to the cluster of~$(x, s)$ if there are times and sites
 $$ s = t_0 < t_1 < \cdots < t_{k - 1} \leq t_k = t \quad \hbox{and} \quad x = x_1, x_2, \ldots, x_k = y $$
 such that the following two conditions hold:
\begin{itemize}
 \item For~$j = 1, 2, \ldots, k$, we have~$\xi_r (x_j) = 1$ for all times~$r \in [t_{j - 1}, t_j]$. \vspace*{4pt}
 \item For~$j = 1, 2, \ldots, k - 1$, there is a~$0 \to 1$ birth mark along the edge~$(x_j, x_{j + 1})$ at time~$t_j$.
\end{itemize}
 In contrast to the definition of descendants, it is permitted to have~$\xi_{t_j^-} (x_{j + 1}) = 1$. \\
\indent If~$\xi_s (x) = 1$ then, for~$t \geq s$, let~$B_t (x, s)$ denote the cluster of~$(x, s)$ at time~$t$, that is,
 $$ B_t (x, s) := \{y \in \Z : (y, t) \ \hbox{is in the cluster of} \ (x, s) \}, $$
 and denote it~$B_t (x)$ for~$s = 0$.
 Again, since interactions are nearest-neighbor,
 $$ \xi_t(y) \neq 2 \quad \hbox{for all} \quad y \in [\inf B_t (x, s), \sup B_t (x, s)] \quad \hbox{and} \quad t \geq s. $$
 As a warm-up to~\eqref{eq:1D-part2}, we prove the following.
\begin{lemma} --
\label{lemma8}
 Let~$\ell_t = \inf B_t (0)$ and~$r_t = \sup B_t (0)$, and let
 $$ \tau = \inf \,\{t > 0 : \xi_t (\ell_t - 1) = 2 \ \hbox{or} \ \xi_t (r_t + 1) = 2 \} .$$
 Then, there are~$p, c, C > 0$ such that
 $$ P(\tau = \infty) \geq p \quad \hbox{and} \quad P (t < \tau < \infty) \leq C \,e^{-ct} $$
 uniformly over~$\xi_0$ such that~$\xi_0(0) = 1$.
\end{lemma}
\begin{proof}
%  As in the proof of Proposition \ref{prop5},
 Define a pair of independent copies~$\xi_t^1$ and~$\xi_t^2$ of the generalized stacked contact process with initial configurations
 $$ \xi_0^1 (y) = \xi_0 (y) \,\ind \{y = 0 \} \quad \hbox{and} \quad \xi_0^2 (y) = \xi_0 (y) \,\ind \{y \neq 0 \}. $$
 Define also
 $$ \begin{array}{rcl}
    \ell_t^1 = \inf \,\{x : \xi_t^1 (x) = 1 \} & \hbox{and} & r_t^1 = \sup \,\{x : \xi_t^1 (x) = 1 \} \vspace*{4pt} \\
       a_t^2 = \sup \,\{x < \ell_t^1 : \xi_t^2 (x) = 2 \} & \hbox{and} & b_t^2 = \inf \,\{x > r_t^1 : \xi_t^2 (x) = 2 \} \end{array} $$
 so that~$\tau$ can also be expressed as
 $$ \tau = \inf \,\{t > 0 : \ell_t^1 - a_t^2 \leq 1 \ \hbox{or} \ b_t^2 - r_t^1 \leq 1 \}. $$
 To show that~$\tau = \infty$ with positive probability, first we fix~$n$ and consider the case~$\min (|a_0^2|, |b_0^2|) \geq n$.
 Using large deviations estimates for the Poisson distribution, we can show that
 $$ P (\max (|\ell_t^1|, |r_t^1|) \geq 2 \lambda_{10} \,t + n / 2 - 1 \ \hbox{for some} \ t > 0) \leq C \,e^{-cn}. $$
 To do so, it suffices to first make an estimate for~$t \leq m_0 := \lfloor n / 4 \lambda_{10} \rfloor$, then for~$t \in [m, m + 1]$ for each~$m \geq m_0$, then to take a union bound.
 Then, taking~$\mu = 2 \lambda_{10}$ and the same~$n$ as in the statement of Lemma~\ref{lemma7}, we find that
 $$ P (\min (|a_t^2|, |b_t^2|) \leq 2 \lambda_{10} \,t + n/2 \ \hbox{for some} \ t > 0) \leq C \,e^{-cn}. $$
 Since in addition
 $$ \min \,(|a_t^2|, |b_t^2|) > \max (|\ell_t^1|, |r_t^1|) \ \hbox{for all} \ t > 0 \quad \hbox{implies that} \quad \tau = \infty, $$
 taking~$n$ large enough, we find that if~$\xi (x) \neq 2$ for~$|x| \leq n$ then
 $$ P (\tau = \infty) \geq 1/2. $$
 For~$\xi$ such that~$\xi (0) = 1$, the probability
 $$ P (\xi_1 (0) = 1 \ \hbox{and} \ \xi_1 (x) = 0 \ \hbox{for all} \ 0 < |x| \leq n) $$
 is at least the probability that, on the time interval~$[0, 1]$, there are no birth marks along edges touching~$[-n, n]$, there is no death mark at~0, and there is a death mark at every~$x$ with~$0 < |x| \leq n$,
 and this probability is at least~$2p$ for some~$p > 0$.
 Using the Markov property and the estimate on~$\tau = \infty$ in the previous case then gives~$P (\tau = \infty) \geq p > 0$ as desired. \\
\indent To deduce the estimate on~$P (t < \tau < \infty)$, we note that
 $$ P (\max (|\ell_s^1|, |r_s^1|) \geq 2 \lambda_{10} \,s \ \hbox{for some} \ s \geq t) \leq C \,e^{-ct} $$
 which can be proved by applying an estimate at each integer time~$n > t$ and summing over~$n$.
 Then, combining with the first statement in Proposition~\ref{prop7} and noting that
 $$ t < \tau < \infty \quad \hbox{implies that} \quad \max (|\ell_s^1|, |r_s^1|) \geq \min (|a_s^2|, |b_s^2|) - 1 \ \hbox{for some} \ s > t, $$
 we deduce the estimate on~$P (t < \tau < \infty)$.
\end{proof} \\ \\
 We are now ready to establish~\eqref{eq:1D-part2} which states the existence of a linearly growing region starting from a random space-time point in which the process agrees with the contact process with parameter~$\lambda_{10}$.
 This will also complete the proof of Theorem~\ref{th:1D}. \\ \\
% \begin{proposition} --
% \label{prop6}
% \blue{Let $\xi_t$ denote the SCP with $\delta=0$ and $\lambda_{20}<\lambda_c<\lambda_{10}$ and let $\zeta_t$ denote the contact process with parameter $\lambda_{10}$ and $\zeta_0(x)=1$ for all $x$ (equivalently, the SCP with $\xi_0(x)=1$ for all $x$).
%  Let $\alpha>0$ denote the edge speed in the contact process with parameter $\lambda_{10}$ [durr]. Then, there exist a (random) site $X$ and a time $T<\infty$ that depend on $\xi$, and site-valued processes $\ell_t \leq r_t$ defined for $t \geq T$ and satisfying $\ell_T = r_T = X$, so that if $\xi$ is a configuration such that $\{x:\xi(x)=1\}$ is an infinite set, then $P$-almost surely,
%  \begin{enumerate}[(i)]
%  \item $\lim_{t\to\infty}-\ell_t/t=\lim_{t\to\infty}r_t/t = \alpha$ and
%  \item $\xi_t(x) = \zeta_t(x)$ for all $x \in [\ell_t,r_t]$ and $t\geq T$.
%  \end{enumerate}}
% \end{proposition}
\begin{demo}{\eqref{eq:1D-part2}}
 Given~$(x, s)$, for~$t \geq s$ recall that~$B_t (x, s)$ denotes the cluster of~$(x, s)$ at time~$t$.
 Let~$\tau_0 = 0$ and~$x_0$ be any site with~$\xi (x_0) = 1$.
 Without loss of generality, suppose that the set~$\{x > 0 : \xi (x) = 1 \}$ is infinite, and define~$x_i$ and~$\tau_i$ recursively for~$\lambda_{21} < \infty$ by letting
 $$ \begin{array}{rcl}
        \ell_t^i & = & \inf \,B_t (x_i, \tau_i) \vspace*{3pt} \\
           r_t^i & = & \sup \,B_t (x_i, \tau_i) \vspace*{3pt} \\
    \tau_{i + 1} & = & \inf \,\{t > \tau_i : \xi_t (\ell_t^i - 1) = 2 \ \hbox{or} \ \xi_t (r_t^i + 1) = 2 \} \vspace*{3pt} \\
       x_{i + 1} & = & \inf \,\{x > x_i : \xi_{\tau_{i + 1}} (x) = 1 \} \end{array} $$
 with the value of~$x_i$ being unimportant if~$\tau_i = \infty$.
 Note that if time~$\tau_i < \infty$ then site~$x_{i + 1}$ is well-defined due to the fact that
 $$ \{\xi : \xi (x) = 1 \ \hbox{for infinitely many} \ x > 0 \} $$
 is an invariant set for the dynamics.
 Let~$N = \sup \,\{i : \tau_i < \infty \}$.
 Applying the strong Markov property and using Lemma~\ref{lemma8}, we obtain that~$N$ is at most geometric with parameter~$p$.
 In addition, by the second part of Lemma~\ref{lemma8}, for~$i = 0, 1, 2, \ldots$,
 $$ \tau_{i + 1} - \tau_i \preceq T_i \quad \hbox{where} \quad P (T_i > t) \leq \max \,(1, C e^{-ct}) $$
 and the random variables~$T_i$ are independent.
 In particular, $\tau_N$ is almost surely finite.
 Let~$T = \tau_N$ and let~$X = x_N$, and let~$\ell_t = \ell_t^N$ and~$r_t = r_t^N$. \\
\indent Recall that~$\zeta_t$ denotes the process with initial configuration~$\zeta_0 (x) = 1$ for all~$x$.
 Since~$\lambda_{10} > \lambda_{20}$ by assumption, a straightforward coupling argument shows that, for any configuration~$\xi_0$,
 $$ \{x \in \Z : \xi_t (x) \neq 0 \} \subseteq \{x \in \Z : \zeta (x) = 1 \}. $$
 Therefore, $\zeta_{\tau_i} (x_i) = 1$ whenever~$\tau_i < \infty$.
 By definition of time~$\tau_{i + 1}$, the set~$B_t(x_i,\tau_i)$ is the set of infected sites in a (single-type) contact process started from the single infected site~$x_i$ at time~$\tau_i$, so a coupling of~\cite{durrett_1980} shows that if~$\tau_i < \infty$ then
 $$ \xi_t (x) = \zeta_t (x) \ \hbox{for all} \ x \in [\ell_t^i, r_t^i] \ \hbox{and all} \ \tau_i < t < \tau_{i + 1}. $$
 In~\cite{durrett_1980}, it is shown that, for the contact process (which in this context means in the absence of any interaction with~2s), $-\ell_t^i/ t$ and~$r_t^i / t \to \alpha > 0$ so the same is true here provided~$\tau_{i + 1} = \infty$, which is
 the case for~$i = N$.
 The proof is now complete.
\end{demo}

%%%%%%%%%%%%%%%%%%%%%%%%%%%%%%%%%%%%%%%%%%%%%%%%%%%%%%%%%%%%%%%%%%%%%%%%%%%%%%%%%%%%%%%%%%%%%%%%%%%%%%%%%%%%%%%%%%%%%%%%%%%%%%%%%%%%%%%%%%

\end{document}